\documentclass[11pt, oneside]{article}
\usepackage{amsfonts}
\usepackage{mathrsfs}
\usepackage{amsfonts}
\usepackage{float}
\usepackage{latexsym,amssymb,amsmath}
\usepackage{latexsym}
\usepackage{amssymb}
\usepackage{amsmath}
\usepackage{amsthm}
\usepackage{indentfirst}
\usepackage{color}
\usepackage{graphicx}
\numberwithin{equation}{section} \allowdisplaybreaks

 \makeatletter
   \renewcommand\@cite[1]{#1\hspace{0.2em}}
   \makeatother

\newtheorem{theorem}{\color{black}\indent Theorem}[section]
\newtheorem{lemma}{\color{black}\indent Lemma}[section]

\newtheorem{remark}{\color{black}\indent Remark}[section]

\usepackage{amssymb,amsmath}
\begin{document}	
\title{\LARGE\bf  Optimal $L^2$-blowup estimates of the Fractional Wave Equation
	\thanks{ The research is  supported by a JST CREST Grant (Number JPMJCR1913, Japan) and a JSPS Grant-in-Aid for Scientific Research (C) (No.23K03174)} \\
	\author{$\rm{Masahiro~Ikeda^{1}}$,  $\rm{~Jinhong~Zhao^{2}}$\thanks{Corresponding author\newline \hspace*{4.5mm} {\it Short running title:} Fractional Wave Equation  \newline \hspace*{4.2mm} 
			{\it Email
				addresses:} masahiro.ikeda@riken.jp~(Masahiro Ikeda), 
			jhzhao23@mails.jlu.edu.cn~(Jinhong~Zhao)
			\newline \hspace*{4.5mm}
			{\it Availability of data and material:} No data and material were used to support this study.  \newline \hspace*{4mm} {\it Competing interests:} The authors declare that there are no conflicts of interest regarding the publication of this paper. }
		\\
	1. Center for Advanced Intelligence Project, RIKEN\\
	2. School of Mathematics, Jilin University, Changchun 130012, PR China}}
\date{} \maketitle

{\bf Abstract:}
This article deals with the behavior in time of the solution to the Cauchy problem for a fractional wave equation with a weighted $L^1$ initial data.
Initially, we establish the global existence of the solution using Fourier methods and provide upper bounds for the $L^2$ norm and the $H^s$ norm of the solution for any dimension $n\in \mathbb{N}$ and $s\in (0,1)$. 
However, when $n=1$ and $s \in [\frac{1}{2},1)$, 
we have to impose a stronger assumption $\int_{\mathbb{R}}u_1(x)dx=0$. 
To remove this stronger assumption, we further use the Fourier splitting method, which yields the optimal blow-up rate for the $L^2$ norm of the solutions. 
Specifically, when $n=1$, the optimal blow-up rate is $t^{1-\frac{1}{2s}}$ for $s \in (\frac{1}{2},1)$ and $\sqrt{\log t}$ for $s = \frac{1}{2}$.

\maketitle


\thispagestyle{empty}

	\section{Introduction}
	In this paper, we study the following Cauthy problem of the  wave equation with  fractional Laplacian
	\begin{equation}\label{1.1}
	\begin{cases}
	u_{tt}+ (- \Delta)^s u=0,~~&(t,x) \in (0,\infty)\times \mathbb{R}^n,\\
	u(0 ,x)=u_{0}(x),~u_t(0 ,x)=u_{1}(x),~~&x \in \mathbb{R}^n,
	\end{cases}
	\end{equation}
	where $s \in (0,1) $ represents  the strength of diffusion and $n \in \mathbb{N}$ denotes the spatial dimension.
	and	we assume the initial data satisfy $u_0\in H^s(\mathbb{R}^n)$ and $u_1\in L^2(\mathbb{R}^n)$, where $H^s(\mathbb{R}^n)$ denotes the $s$-th order $L^2$-based Sobolev space.

	Before investigating problem \eqref{1.1}, we recall previous works on the asymptotic properties of solutions to the Cauchy problem for wave equations as time goes to infinity.
	Many mathematicians investigated  the following general form:
	\begin{equation}\label{22}
	\begin{cases}
	u_{tt}+ (- \Delta)^s u+a(- \Delta)^\theta u_t=0,~~&(t,x) \in (0,\infty)\times \mathbb{R}^n,\\
	u(0 ,x)=u_{0}(x),~u_t(0 ,x)=u_{1}(x),~~&x \in \mathbb{R}^n,
	\end{cases}
	\end{equation}
	where $s \ge 0 $, $\theta  \ge 0$ and $a\ge 0$.
	When $s=1$,  $a=1$ and $\theta = 1$, the pioneering works of Ponce [\ref{GP}] and Shibata [\ref{YS}] are well-known, where they studied various $L^p$-$L^q$  estimates of  the solution  to problem \eqref{22}.
	Subsequently, Ikehata [\ref{RIA}] introduced the asymptotic profiles of  the solution, and on the basis of [\ref{RIA}], Ikehata-Onodera [\ref{remark}] further provided optimal blow-up estimates for the $L^2$  norm of the solution as $t\to \infty$ in low-dimensional cases, where the optimal blow-up rate is $\sqrt{t}$ for $n=1$ and $\sqrt{\log t}$ for $n=2$.
	Additionally, when $\theta = 0$, Karch [\ref{GK}] presented a self-similar profile of the solution in an asymptotic sense as $t\to \infty$.
	Furthermore, in [\ref{TFR}], Fukushima-Ikehata-Michihisa similarly provided some optimal estimates in the low frequency region when $\theta = 2$.

	Meanwhile, for the case of $a=1$ and $s=\theta\ge1$, Ikehata [\ref{note}] obtained asymptotic profiles of the solution with weighted $L ^{1,1} (\mathbb{R}^n )$ initial data  and studied optimal estimates of the solution in the $L^2$-sense, for the definition of $L^{1,1}(\mathbb{R}^n)$ see (\ref{weighted L^1}). Later, Fujiwara-Ikeda-Wakasugi [\ref{KFM}] investigated a second-order evolution equation with fractional Laplacian and damping for data, where $ a=1, ~s \ge 0, ~\theta = 0$, and provided time decay estimates for the solution in weighted Sobolev spaces.

	Recently, Ikehata studied optimal $L^2$ norm estimates for the solution of free wave [\ref{RIL}] and free plate [\ref{RIL2}]. Specifically, when $s=1$ and $a=0$, Ikehata [\ref{RIL}] obtained sharp infinite time blow-up estimates for the $L^2$ norm of the solution in the cases of $n = 1$ and $n = 2$ when the $0$-th moment of the initial velocity does not vanish. This was then applied to local energy decay estimates for the case of $n = 2$.
	When $s=2$ and $a=0$, Ikehata [\ref{RIL2}] also obtained similar results for $n \le 4$. It was then conjectured that for more general $\sigma$-evolution equations, i.e. $s=\sigma$ and the critical number $n^*$ on the dimension $n$ for the blow-up phenomenon can be defined as $n^* = 2\sigma$.
	However, some questions have not been resolved in past studies.
	
	$\bullet$ When $s \in (0,1)$, does a unique global weak solution exist for the problem \eqref{1.1}?
	
	$\bullet$ How should we address the failure of Hardy-type inequality and the Poincar$\rm{\acute{e}}$ inequality to obtain the $L^2$ boundedness of the solution, particularly in low-dimensional case?
	
	$\bullet$ When $s \in (0,1)$, if the problem \eqref{1.1} does not have a unique global weak solution, does the $L^2$ norm of the solution exhibit optimal blow-up? Is it possible to obtain the optimal blow-up rate?

	In this paper, we give a positive answer to the questions above.
	To be precise, we first use the Fourier method to provide a formal solution to problem \eqref{1.1} and prove the global existence of the solution when $s \in(0,1)$, while also demonstrating the boundedness of the $L^2$ norm and the $H^s$ norm of the solution. However, when $n=1$ and $s \in[\frac{1}{2},1)$, we require the initial velocity to belong to a weighted function space and to satisfy $\int_{\mathbb{R}}u_1(x)dx=0$.

	Next, we remove the above condition and consider the case where the 0-th moment of the initial velocity does not vanish. By utilizing the Fourier splitting method, we select appropriate time-dependent functions to partition the frequency domain. Through further investigation of the low frequency region, we provide a lower bound for the blow-up of the $L^2$ norm of the solution when $n=1$ and $s \in (\frac{1}{2},1)$. However, when $s = \frac{1}{2}$, the above method fails. In this case, we apply a more refined area analysis method over the entire frequency domain and similarly obtain a lower bound for the blow-up of the  $L^2$ norm of the solution.
	
	Finally, we make a more detailed division of the frequency domain, and when $n=1$ and $s \in [\frac{1}{2},1)$, we derive an upper bound for the blow-up of the $L^2$ norm of the solution.
	And then we obtain the optimal infinite time blow-up results.

	Our main results read as follows.
	First, we present the existence, uniqueness, and boundedness results for the solution of problem \eqref{1.1}.

	\begin{theorem}\label{exist1}
		Let $n \ge 2$ and $s \in (0,1)$, or $n =1$ and $s\in (0,\frac{1}{2})$. 	For each initial data $[u_0,u_1]\in H^s(\mathbb{R}^n) \times (L^2(\mathbb{R}^n)\cap L^1(\mathbb{R}^n)),$ there exists a unique solution $u \in C([0, \infty); H^s(\mathbb{R}^n)) \cap C^1([0, \infty); L^2(\mathbb{R}^n)) $ to the problem \eqref{1.1}  satisfying the following estimates:
		$$\|u(t,\cdot)\|_{2} \le \sqrt{2}\|u_0\|_2+ C(\|u_1\|_2+\|u_1\|_1),$$
		and
		$$\|u(t,\cdot)\|_{H^s(\mathbb{R}^n)} \le \sqrt{2}\|u_0\|_{H^s(\mathbb{R}^n)}+ C(\|u_1\|_2+\|u_1\|_1),$$
		for $t \ge 0$, with some constant $C=C_{n,s}>0,$ and satisfying the energy conservation property:
		$$E(t) = E(0), ~~~~t \ge 0,	$$
		where the total energy $E(t)$ to the problem \eqref{1.1} is defined by
		$$E(t) =\frac{1}{2}\left( \int_{\mathbb{R}^n} |u_t|^2 dx +  \int_{\mathbb{R}^n} |(-\Delta)^{\frac{s}{2}} u|^2 dx \right).$$		
	\end{theorem}	
	
	The following theorem extends the previous results to the one-dimensional case with specific conditions on the initial data, including the zero-mean condition on the initial velocity $u_1$. It ensures the existence and uniqueness of the solution while providing similar upper bounds for the $L^2$ norm and $H^s$ norm of the solution.

	\begin{theorem}\label{bound2}
		Let $n = 1$  and $s\in [\frac{1}{2},1)$.  If $[u_0,u_1]\in H^s(\mathbb{R}) \times (L^2(\mathbb{R})\cap L^{1,\gamma}(\mathbb{R})),$ for all $ \gamma \in [s-\frac{1}{2},1]$, and further satisfies
		$$\int_{\mathbb{R}}u_1(x)dx=0,$$
		then there exists a unique solution $u \in C([0, \infty); H^s(\mathbb{R})) \cap C^1([0, \infty); L^2(\mathbb{R})) $ to the problem \eqref{1.1}, satisfying
		$$\|u(t,\cdot)\|_2 \le \sqrt{2}\|u_0\|_2+ C(\|u_1\|_2+\|u_1\|_{1,\gamma}),$$
		and
		$$\|u(t,\cdot)\|_{H^s(\mathbb{R})} \le \sqrt{2}\|u_0\|_{H^s(\mathbb{R})}+ C(\|u_1\|_2+\|u_1\|_{1,\gamma}),$$
		for $t \ge 0$, with some constant $C=C_{s,\gamma}>0.$
	\end{theorem}

	Next, we remove the zero-mean condition on the initial velocity to investigate the long-term behavior of the solution, leading to optimal blow-up results for the $L^2$ norm of the solution as time approaches infinity.

	\begin{theorem}\label{th1}
		Let $n = 1,$ $s\in (\frac{1}{2},1).$ For any initial data $[u_0, u_1]\in H^s(\mathbb{R})\times L^2(\mathbb{R})$, the solution $u(t, x)$ to problem \eqref{1.1} satisfies the following properties under the additional regularity conditions on the initial data:
		
		{\rm (1)} If $u_1 \in L^1(\mathbb{R})$, then $$\|u(t)\|_2\le \left( \frac{4s}{2s-1}   \right) ^{\frac{1}{2}}M_1t^{1-\frac{1}{2s}}.   $$
		
		{\rm (2)}  If $u_1 \in L^{1,s}(\mathbb{R})$, then $$\frac{\theta_0}{4} \left| \int_{\mathbb{R}} u_1(x) dx\right| t^{1-\frac{1}{2s}} \le \|u(t)\|_2,$$
		for $t \gg 1,$ where $\theta_0 \in (0, 1)$  is a constant and $M_1=\|u_0\|_2+\|u_1\|_1.$
	\end{theorem}
	
	However, for the critical case $s=\frac{1}{2}$, Theorem \ref{th1} does not yield the corresponding optimal blow-up results of the solution, and Theorem \ref{th2} below gives a positive answer.
	
	\begin{theorem}\label{th2}
		Let $n = 1,$ $s=\frac{1}{2}.$ For any initial data $[u_0, u_1]\in H^s(\mathbb{R})\times L^2(\mathbb{R})$, the solution $u(t, x)$ to problem \eqref{1.1} satisfies the following properties under the additional regularity conditions on the initial data:
		
		{\rm (1)} If $u_1 \in L^1(\mathbb{R})$, then $$\|u(t,\cdot)\|_2\le 2M_{2}\sqrt{\log t}. $$
		
		{\rm (2)}  If $u_1 \in L^{1,1}(\mathbb{R})$, then $$  \frac{1}{3 e} \left| \int_{\mathbb{R}} u_1(x) dx\right| \sqrt{\log t} \le \|u(t,\cdot)\|_2,$$
		for $t \gg 1,$ where $M_2=\|u_0\|_2+\|u_1\|_2+\|u_1\|_1.$	
	\end{theorem}

	\begin{remark}
		\rm
		The results of the low-dimensional case (i.e., $n$ = 1 and $s \in [\frac{1}{2},1)$) are optimal, and imply the blow-up properties in infinite time  (as $t \to \infty$) of the $L^2$ norm of the solution to problem \eqref{1.1}.
		Specifically,
		When $n=1$ and $s \in(\frac{1}{2},1)$, the optimal blow-up rate of the $L^2$ norm of the solution is $t^{1-\frac{1}{2s}}$. While when $n=1$ and $s = \frac{1}{2}$,  the optimal blow-up rate is $\sqrt{ \log t}$.
		If one chooses $s = 1$, the obtained results completely coincide with that of [\ref{RIL}].
		As far as we know, there does not appear to be any previous work directly investigating the optimal blow-up properties of the solution for problem \eqref{1.1}.		
	\end{remark}

	This paper is organized as follows. In Section 2,
	we give some notations, definitions and three important lemmas used in the proofs.  Sections 3 and 4 will be devoted to the boundedness results Theorem \ref{exist1}, \ref{bound2}, and the optimal blow-up rate results Theorem \ref{th1}, \ref{th2}, respectively.

	\section{Preliminaries}
	In this section, we first introduce some notations and definitions that will be used throughout the paper.
	In what follows, we denote by $\| \cdot \|_r~(r \ge 1) $ the norm in $L^r(\mathbb{R}^n)$.
	$C$ denotes a generic positive constant, which may differ at each appearance.

	Following [\ref{hit}], for any the fractional exponent $s \in (0, 1)$, we define the function space $H^s(\mathbb{R}^n)$  as follows	
	$$H^s(\mathbb{R}^n):=  \left\lbrace u \in L^2(\mathbb{R}^n) ~\Big|~ \frac{|u(x)-u(y)|}{|x-y|^{\frac{n}{2}+s}} \in L^2(\mathbb{R}^n \times \mathbb{R}^n) \right\rbrace,$$	
	i.e., an intermediary Banach space between $L^2(\mathbb{R}^n)$ and $H^1(\mathbb{R}^n)$, endowed with the natural norm	
	$$\|u\|_{H^s(\mathbb{R}^n)} := \left(\int_{\mathbb{R}^n} |u|^2 dx + \int_{\mathbb{R}^n} \int_{\mathbb{R}^n} \frac{|u(x)-u(y)|^2}{|x-y|^{n+2s}} dx dy \right) ^{\frac{1}{2}},$$	
	where the term	
	$$\|u\|_{\dot{H}^s(\mathbb{R}^n)} := [u]_{H^s(\mathbb{R}^n)}
	= \left( \int_{\mathbb{R}^n} \int_{\mathbb{R}^n} \frac{|u(x)-u(y)|^2}{|x-y|^{n+2s}} dx dy  \right) ^{\frac{1}{2}}
	=\sqrt{2} C(n,s)^{-\frac{1}{2}}\|(-\Delta)^{\frac{s}{2}}u\|_2,$$
	where
	$$C(n,s)=\left(\int_{\mathbb{R}^n} \frac{1-\cos(\zeta_1)}{|\zeta|^{n+2s}}d\zeta \right) ^{-1} <+\infty,$$
	and the fractional operators $(-\Delta)^s~\left( s \in (0, 1)\right) $ be defined by (formally)	
	$$\left((-\Delta)^s u \right) (x) : = \mathscr{F}^{-1}_{\xi \to x}\left(|\xi|^{2s} \mathscr{F}(u)  \right) (x),$$	
	where $\mathscr{F}_{x\to\xi}(f)(\xi)$ denote the Fourier transform of $f$ defined by
	\begin{equation}\label{f1}
	\mathscr{F}_{x\to\xi}(f)(\xi):=\hat{f}(\xi)=\int_{\mathbb{R}^n}e^{-ix\cdot \xi}f(x)dx,~~~~\xi \in \mathbb{R}^n,
	\end{equation}
	and the inverse Fourier transform of $f$ is defined by
	\begin{equation}\label{f2}
	\mathscr{F}^{-1}_{\xi\to x}(f)(x):=(2\pi)^{-n}\int_{\mathbb{R}^n}e^{ix\cdot \xi}f(\xi)d\xi,~~~~x \in \mathbb{R}^n,
	\end{equation}
	as usual with $i := \sqrt{-1}$.

	In addition, we also introduce
	the following weighted functional spaces
	\begin{equation}
	\label{weighted L^1}
	L^{1,\gamma}(\mathbb{R}^n):=\left\lbrace f\in L^1(\mathbb{R}^n)~\Big|~ \|f\|_{1,\gamma}:=\int_{\mathbb{R}^n}\left(1+ |x|^\gamma\right)|f(x)| dx <+\infty \right\rbrace. 
	\end{equation}

	Next, we give some important lemmas. Lemmas \ref{fno0} and \ref{f=0} from [\cite{RIF}, Proposition 2.1] give the inequalities concerning the Fourier image of the Riesz potential.

	\begin{lemma}$^{[\ref{RIF}]}$\label{fno0}
		Let  $n \ge 1$ and $\theta \in  [0, \frac{n}{2})$. Then, for all $f\in L^2(\mathbb{R}^n) \cap L^{1}(\mathbb{R}^n)$
		it is true that
		$$\int_{\mathbb{R}^n}\frac{|\hat{f}(\xi)|^2}{|\xi|^{2\theta}}d\xi\le C\left(\|f\|^2_{1}+\|f\|_2^2 \right), $$
		with some constant $C>0$, which depends on $n$ and $\theta$.		
	\end{lemma}
	
	\begin{lemma}$^{[\ref{RIF}]}$\label{f=0}
		Let $n \ge 1$, $\gamma \in [0, 1]$ and $\theta \in  [0,\gamma+\frac{n}{2})$. Then, for all $f\in L^2(\mathbb{R}^n) \cap L^{1,\gamma}(\mathbb{R}^n)$ satisfying
		$$\int_{\mathbb{R}^n}f(x)dx=0,$$
		it is true that
		$$\int_{\mathbb{R}^n}\frac{|\hat{f}(\xi)|^2}{|\xi|^{2\theta}}d\xi\le C\left(\|f\|^2_{1,\gamma}+\|f\|_2^2 \right), $$
		with some constant $C>0$, which depends on $n,\theta$ and $\gamma$.				
	\end{lemma}

	The next lemma plays an important role in proving the blow-up estimates of the solution, which is described as follows.	
	\begin{lemma}$^{[\ref{RIN}]}$\label{xi}
		Let $\gamma \in [0,1]$ and $f\in L^{1,\gamma}(\mathbb{R}^n).$ Then it holds that
		\begin{equation}\label{AB}
		\left|\mathscr{F}(f)(\xi) \right| \le  C_\gamma |\xi|^{\gamma}\|u_1\|_{{1,\gamma}}+\left|\int_{\mathbb{R}^n} f(x)dx \right|,
		\end{equation}
		with some constant $C_\gamma>0$, which depends only on $\gamma$.
	\end{lemma}

	\section{The boundedness results}
	
	\noindent{\bf Proof of Theorem \ref{exist1} and \ref{bound2}.}
	By taking the Fourier transform $\mathscr{F}$ of the problem \eqref{1.1}, we see that there exists a formal solution $u = u(t, x)$ to the problem (1.1) with suitable initial data $(u_0, u_1)$ such that the identity
	\begin{equation}\label{uw}
	u(t,x)=\left\lbrace \mathcal{W}(t)u_1 \right\rbrace (x) +\left\lbrace \partial_t\mathcal{W}(t)u_0 \right\rbrace (x),
	\end{equation}
	holds for any $(t, x) \in [0,\infty) \times \mathbb{R}^n.$ Here $\mathcal{W}(t)$ on $[0,\infty)$ is the solution operator defined by
	$$\mathcal{W}(t):= \mathscr{F}^{-1}R_1(t,\xi) \mathscr{F},~~~R_1(t,\xi)=\frac{\sin\left(t|\xi|^s \right) }{|\xi|^s},$$	
	here $\mathscr{F}$ and $\mathscr{F}^{-1}$ denote the Fourier transform and the inverse Fourier transform,
	respectively (see \eqref{f1} and \eqref{f2} for the definitions).

	Note that  $[u_0,u_1]\in H^s(\mathbb{R}^n) \times L^2(\mathbb{R}^n).$
	According to the Plancherel theorem and \eqref{uw},  we get
	\begin{align}\label{u2.1}
	\|u\|^2_{2}=\|\hat{u}\|^2_{2}
	&=\int_{\mathbb{R}^n}\left| \mathscr{F}\left( \mathcal{W}(t)u_1 + \partial_t\mathcal{W}(t)u_0 \right)\right| ^2d \xi \nonumber\\
	&\le 2\int_{\mathbb{R}^n}\left|\mathscr{F}\left( \mathcal{W}(t)u_1\right) \right| ^2d \xi    + 2\int_{\mathbb{R}^n}\left|\mathscr{F} \left(  \partial_t\mathcal{W}(t)u_0 \right) \right| ^2d \xi    \nonumber\\	
	&\le 2\int_{\mathbb{R}^n}\left|\frac{\sin\left(t|\xi|^s \right) }{|\xi|^s}\hat{u}_1(\xi) \right| ^2d \xi  + 2\int_{\mathbb{R}^n}\left|\cos\left(t|\xi|^s \right)\hat{u}_0(\xi)\right| ^2d \xi   \nonumber\\	
	&\le 2\int_{\mathbb{R}^n}\left|\frac{ \hat{u}_1(\xi)}{|\xi|^s} \right| ^2d \xi    + 2\|u_0\|^2_{2}.
	\end{align}
	Then,  by using Lemma \ref{fno0}, when $n \ge 2$  $\left( s\in (0,1)\right) $ or $n =1$ $\left( s\in (0,\frac{1}{2})\right) $, one has
	\begin{equation}\label{v5}
	\int_{\mathbb{R}^n_\xi}\frac{\left| \hat{u_1}(\xi)\right|^2}{|\xi|^{2s}}d\xi
	\le C (\|u_1\|_2^2+\|u_1\|_1^2).
	\end{equation}
	In the case when $n =1$ and $s\in [\frac{1}{2},1)$, relying on Lemma \ref{f=0} with $\gamma\in [s-\frac{1}{2},1)$ for $n = 1$, one can obtain
	\begin{equation}\label{v6}
	\int_{\mathbb{R}^n_\xi}\frac{\left| \hat{u_1}(\xi)\right|^2}{|\xi|^{2s}}d\xi
	\le C (\|u_1\|_2^2+\|u_1\|_{1,\gamma}^2),
	\end{equation}
	where one has just used the assumption $\int_{\mathbb{R}^n}u_1(x)dx=0$ in Lemma \ref{f=0}. Therefore, it follows from \eqref{u2.1}, \eqref{v5} or \eqref{v6} that
	\begin{align}\label{u2.3}
	\|u\|^2_{2}
	\le   2\|u_0\|^2_{2} + C (\|u_1\|_2^2+\|u_1\|_1^2),~~~t \ge 0,
	\end{align}
	where $n\ge 2~\text{and}~s\in (0,1)~\text{or}~n =1~\text{and}~s\in (0,\frac{1}{2})$. And
	\begin{align}\label{u2.4}
	\|u\|^2_{2}
	\le   2\|u_0\|^2_{2} + C (\|u_1\|_2^2+\|u_1\|_{1,\gamma}^2),~~~t \ge 0,
	\end{align}
	where $n=1~\text{and}~s\in [\frac{1}{2},1)$.
	
	
	Next, from the definition of  $\|u\|_{\dot{H}^s(\mathbb{R}^n)}$ and \eqref{uw}, we arrive at
	\begin{align}\label{u2.2}
	\|u\|^2_{\dot{H}^s(\mathbb{R}^n)}
	&= 2 C(n,s)^{-1}\|(-\Delta)^{\frac{s}{2}}u\|^2_{2}
	= 2 C(n,s)^{-1}\left\||\xi|^{s} \hat{u}\right\|^2_{2} \nonumber\\	
	&\le 2 C(n,s)^{-1}	\int_{\mathbb{R}^n}\left||\xi|^{s} \mathscr{F}\left( \mathcal{W}(t)u_1 + \partial_t\mathcal{W}(t)u_0 \right)\right| ^2d \xi \nonumber\\
	&\le 4 C(n,s)^{-1} \left( \int_{\mathbb{R}^n}\left|\sin\left(t|\xi|^s \right) \hat{u}_1(\xi) \right| ^2d \xi  + \int_{\mathbb{R}^n}\left||\xi|^{s} \cos\left(t|\xi|^s \right)\hat{u}_0(\xi)\right| ^2d \xi\right)    \nonumber\\	
	&\le 4 C(n,s)^{-1} \|u_1\|^2_{2}    + 2\|u_0\|^2_{\dot{H}^s(\mathbb{R}^n)}.
	\end{align}
	
	Similarly, we also get
	\begin{align}\label{ut.1}
	\|u_t\|^2_2
	&= \int_{\mathbb{R}^n}\left|\partial_t\mathcal{W}(t)u_1  + \partial_t^2\mathcal{W}(t)u_0\right| ^2d \xi \nonumber\\
	&\le 2\|u_1\|^2_{2}    +2 \left\||\xi|^{s} \hat{u}_0\right\|^2_{2}
	\le 2\|u_1\|^2_{2}    +  C(n,s)\|u_0\|^2_{\dot{H}^s(\mathbb{R}^n)}.
	\end{align}
	Combining \eqref{u2.2} and \eqref{u2.3} or \eqref{u2.4}, we obtain
	\begin{align}\label{u2.5}
	\|u\|^2_{H^s(\mathbb{R}^n)}
	\le   2\|u_0\|^2_{H^s(\mathbb{R}^n)} + C (\|u_1\|_2^2+\|u_1\|_1^2),~~~t \ge 0,
	\end{align}
	where $n\ge 2~\text{and}~s\in (0,1)~\text{or}~n =1~\text{and}~s\in (0,\frac{1}{2})$, with some constant $C>0$, which depends on $n$ and $\theta$.	
	And
	\begin{align}\label{u2.6}
	\|u\|^2_{H^s(\mathbb{R})}
	\le   2\|u_0\|^2_{H^s(\mathbb{R})} +C (\|u_1\|_2^2+\|u_1\|_{1,\gamma}^2),~~~t \ge 0,
	\end{align}
	where $n=1~\text{and}~s\in [\frac{1}{2},1)$, with some constant $C>0$, which depends on $n,~\theta$ and $\gamma$. This completes the proofs of Theorem \ref{exist1} and Theorem \ref{bound2}.

	However, in the case of $n = 1$ and $s \in [\frac{1}{2},1)$, one has to impose stronger assumption such that $I_0=\int_{\mathbb{R}}u_1(x)dx=0$, which means the vanishing condition of the 0-th moment of  the initial velocity. 	
	Next, we use the idea of [\ref{RIL}] to remove this assumption.

	\section{The optimal blow-up rate results}
	In this section, based on the idea of [\ref{RIL}], in the case of $n = 1$, we prove the Theorem \ref{th1}, \ref{th2} and obtain the optimal blow-up estimates of $\|u(t,\cdot)\|_2$  by using the Plancherel Theorem and Fourier splitting method, where the lower bound estimates guarantee the blow-up of the solution and and the upper bound estimates guarantee its optimality.
	
	\subsection{$L^2$-lower bound estimates of the solution}
	In this subsection, let us derive the lower bound estimates of $\|u(t,\cdot)\|_2$
	by using low frequency estimates.
	For the moment, we shall assume that the initial data $[u_0, u_1]$ are sufficiently smooth because the density argument can be applied in the final estimates.
	

	In order to get the lower bound estimates for $\|u(t,\cdot)\|_2$, it suffices to treat $\|u(t,\cdot)\|_2$ with $\|\hat{u}(t,\xi)\|_2$ because of the Plancherel Theorem.
	Then, we divide $\|\hat{u}(t,\xi)\|_2$ into the following  low and high frequency parts  by choosing a suitable time-dependent function
	\begin{equation}\label{lh}
	\|\hat{u}(t,\xi)\|_2^2=\left(\int_{|\xi|\le \theta_0t^{-\frac{1}{s}}} + \int_{|\xi|\ge \theta_0t^{-\frac{1}{s}}} \right)|\hat{u}(t,\xi)|^2d \xi=I_{\text{low}}(t)+I_{\text{high}}(t),
	\end{equation}
	where $\theta_0 \in (0, 1) $ such that
	\begin{equation}\label{sin}
	\left|\frac{\sin \theta}{\theta} \right|\ge \frac{1}{2},
	\end{equation}
	for all $\theta \in (0, \theta_0]$, since
	\begin{equation*}
	\lim\limits_{\theta \to +0}\frac{\sin \theta}{\theta} = 1,
	\end{equation*}
	and we have
	$\theta_0t^{-\frac{1}{s}}\le 1$
	as long as $t>1$ for $s \in (0,1)$.
	
	Reviewing \eqref{uw}, it is easy to know
	\begin{equation}\label{ode1}
	\hat{u}(t,\xi)=\frac{\sin\left(t|\xi|^s \right) }{|\xi|^s}\hat{u}_1(\xi)+\cos\left(t|\xi|^s \right)\hat{u}_0(\xi).
	\end{equation}
	
	In the following, by using low frequency estimates, we get that the lower bound for the $L^2$ norm of the Fourier transform of $u$ (see \eqref{ode1}) blows up  polynomially, which shows that the solution to problem \eqref{1.1} has blow-up properties.
	\begin{lemma}\label{le1}
		Let $n = 1$, $s \in (\frac{1}{2},1)$ and $[u_0, u_1] \in L^2 (\mathbb{R}) \times L^{1,s}(\mathbb{R})$. Then, it holds that
		$$\|\hat{u}(t,\cdot)\|_2\ge \frac{\theta_0}{4} \left|\int_{\mathbb{R}} u_1(x) dx \right|  t^{1-\frac{1}{2s}},~~~~t\gg 1,$$
		where $\theta_0 \in (0, 1)$  is a constant.
	\end{lemma}
	\begin{proof}
		According to (\ref{ode1}), by considering the lower bound for $I_ {\text {low}} (t)$  instead of the lower bound for $\|\hat {u} (t, \cdot) \|_2^2 $, we get
		\begin{align}\label{l1}
		I_{\text{low}}(t)
		&= \int_{|\xi|\le \theta_0t^{-\frac{1}{s}}} \left| \frac{\sin\left(t|\xi|^s \right) }{|\xi|^s}\hat{u}_1(\xi)+\cos\left(t|\xi|^s \right)\hat{u}_0(\xi)\right| ^2 d\xi\nonumber\\
		& \ge \frac{1}{2} \int_{|\xi|\le \theta_0t^{-\frac{1}{s}}} \frac{\sin^2\left(t|\xi|^s \right) }{|\xi|^{2s}}\left|\hat{u}_1(\xi)\right| ^2 d\xi -\int_{|\xi|\le \theta_0t^{-\frac{1}{s}}} \cos^2\left(t|\xi|^s \right) \left|\hat{u}_0(\xi)\right| ^2 d\xi\nonumber\\
		&=:\frac{1}{2}J_1(t)-J_2(t).
		\end{align}

		
		First,  from Lemma \ref{xi}, when $u_1 \in L^{1,s}(\mathbb{R})$,  we have
		\begin{equation}\label{u1}
		\left| \hat{u}_1(\xi)\right| \le C_s |\xi|^s\|u_1\|_{{1,s}}+\left|\int_{\mathbb{R}} u_1 dx \right|,
		\end{equation}
		where  $C_s>0$ depends only on $s$. Then by substituting   \eqref{sin} and \eqref{u1} into $J_1(t)$, we obtain
		\begin{align}\label{l2}
		J_1(t) &= \frac{1}{2} \int_{|\xi|\le \theta_0t^{-\frac{1}{s}}} \frac{\sin^2\left(t|\xi|^s \right) }{|\xi|^{2s}}\left|\hat{u}_1(\xi)\right| ^2 d\xi\nonumber\\
		& \ge \frac{P^2}{16} t^2\int_{|\xi|\le \theta_0t^{-\frac{1}{s}}} d\xi
		- \frac{1}{2}  C_s^2 \|u_1\|_{{1,s}} ^2 \int_{|\xi|\le \theta_0t^{-\frac{1}{s}}}  d\xi \nonumber\\
		& = \frac{\theta_0}{8} P^2 t^{2-\frac{1}{s}}
		- C_s^2   \theta_0 \|u_1\|_{{1,s}} ^2 t^{-\frac{1}{s}},
		\end{align}
		where
		$$P=\int_{\mathbb{R}}u_1(x)dx.$$
		Additionally, it is easy to get
		\begin{align}\label{l3}
		J_2(t) = \int_{|\xi|\le \theta_0t^{-\frac{1}{s}}} \cos^2\left(t|\xi|^s \right) \left|\hat{u}_0(\xi)\right| ^2 d\xi
		\le \int_{|\xi|\le \theta_0t^{-\frac{1}{s}}} \left|\hat{u}_0(\xi)\right| ^2 d\xi
		\le \|u_0\|_2^2.
		\end{align}
		Together with \eqref{l1},  \eqref{l2} and  \eqref{l3}, we  arrive at the following estimate
		\begin{align}
		\|\hat{u}\|_2^2\ge 	I_{\text{low}}(t)
		\ge \frac{\theta_0}{8} P^2 t^{2-\frac{1}{s}}
		- C_s^2   \theta_0 \|u_1\|_{{1,s}} ^2 t^{-\frac{1}{s}}
		-\|u_0\|_2^2.
		\end{align}
		Thus, in the case of $s \in (\frac{1}{2},1)$, there exists a positive real number $t_0$, one has the following blow-up property of the solution when $P \neq 0$
		\begin{align}\label{l}
		\|\hat{u}\|_2^2
		\ge  \frac{\theta_0^2}{16} P^2 t^{2-\frac{1}{s}},
		\end{align}
		for all $t \ge t_0$, where $t_0 > 0$ depends on  $s$,   $\|u_1\|_{{1,s}} $ and $\|u_0\|_{2} $. The proof of the Lemma \ref{le1} is complete.
	\end{proof}
	
	However, when $n=1,~s=\frac{1}{2}$, the above proof can not give the blow-up property of the $L^2$ norm of the solution. 
	In the following, we address this problem by means of new area estimation and other methods, namely Lemma \ref{le2}. 
	
	\begin{lemma}\label{le2}
		Let $n=1,~s=\frac{1}{2}$, and $[u_0, u_1]\in L^2 (\mathbb{R}) \times L^{1,1}(\mathbb{R}).$ Then, it holds that
		\begin{equation}
		\|\hat{u}(t,\cdot)\|_2 \ge \frac{1}{4 e}\left|\int_{\mathbb{R}} u_1(x) dx \right| \sqrt{\log t}.
		\end{equation}
	\end{lemma}
	
	\begin{proof}
		First, similar to \eqref{l1} calculation, we can easily get
		\begin{align}\label{sl1}
		\|\hat{u}\|_2^2
		&= \int_\mathbb{R} \left| \frac{\sin\left(t|\xi|^s \right) }{|\xi|^s}\hat{u}_1(\xi)+\cos\left(t|\xi|^s \right)\hat{u}_0(\xi)\right| ^2 d\xi\nonumber\\
		& \ge \frac{1}{2} \int_\mathbb{R}  \frac{\sin^2\left(t|\xi|^s \right) }{|\xi|^{2s}}   \left|\hat{u}_1(\xi)\right| ^2 d\xi -\int_\mathbb{R} \cos^2\left(t|\xi|^s \right) \left|\hat{u}_0(\xi)\right| ^2 d\xi.
		\end{align}
		The estimation of the second term on the right side of \eqref{sl1} is straightforward. For the estimation of the first term on the right side of \eqref{sl1}, we introduce a trick function $e^{-|\xi|^2}$,  it follows from \eqref{u1} that
		\begin{align}\label{sl2}
		\|\hat{u}\|_2^2
		&\ge \frac{P^2}{4} \int_\mathbb{R}  e^{-|\xi|^2} \frac{\sin^2\left(t|\xi|^{\frac{1}{2}} \right) }{|\xi|} d\xi
		- \frac{C_1^2}{2}  \|u_1\|_{{1,1}} ^2  \int_\mathbb{R}  |\xi| e^{-|\xi|^2} d\xi -\|u_0\|_2^2 \nonumber\\
		&=: \frac{P^2}{4}  K_1(t) - \frac{C_1^2}{2}  \|u_1\|_{{1,1}} ^2 -\|u_0\|_2^2 .
		\end{align}
		
		In the forthcoming step, we are in a position to estimate $K_1(t)$ in \eqref{sl2}, we have
		\begin{align}\label{sl3}
		K_1(t)
		= 2\int_0^{+\infty}   \frac{e^{-r^2} }{r}\sin^2\left(tr^{\frac{1}{2}} \right) dr
		\ge \sum_{i=0}^{\infty}\int_{a_i}^{b_i} \frac{e^{-r^2} }{r} dr,
		\end{align}
		where
		$$a_i=\left[\left( \frac{1}{4}+i\right) \frac{\pi}{t} \right] ^2,~~~~b_i=\left[\left( \frac{3}{4}+i\right) \frac{\pi}{t} \right] ^2,~~~~~~i=0,1,2,\cdots$$
		
		Next, we will prove the following important inequality by means of some area estimation
		\begin{equation}\label{mian1}
		\int_{a_0}^{\infty} \frac{e^{-r^2} }{r} dr \le 3\sum_{i=0}^{\infty}\int_{a_i}^{b_i} \frac{e^{-r^2} }{r} dr.
		\end{equation}
		Indeed, let
		$$A_i=\int_{a_i}^{b_i} \frac{e^{-r^2} }{r} dr,~~~~B_i=\int_{b_i}^{a_{i+1}} \frac{e^{-r^2} }{r} dr,~~~~i=0,1,2,\cdots$$
		Consider the monotone decreasing property of the function $ \frac{e^{-r^2} }{r}$, we obtain
		\begin{equation}\label{mian2}
		A_i\ge \left[ \left[\left( \frac{3}{4}+i\right) \frac{\pi}{t} \right] ^2- \left[\left( \frac{1}{4}+i\right) \frac{\pi}{t} \right] ^2\right]
		\frac{e^{- \left[\left( \frac{3}{4}+i\right) \frac{\pi}{t} \right] ^2}}{ \left[\left( \frac{3}{4}+i\right) \frac{\pi}{t} \right] ^2},
		\end{equation}
		and
		\begin{align}\label{mian3}
		B_i&\le \frac{1}{2}\left[ \left[\left( \frac{5}{4}+i\right) \frac{\pi}{t} \right] ^2- \left[\left( \frac{3}{4}+i\right) \frac{\pi}{t} \right] ^2\right]
		\left[ \frac{e^{- \left[\left( \frac{3}{4}+i\right) \frac{\pi}{t} \right] ^2}}{ \left[\left( \frac{3}{4}+i\right) \frac{\pi}{t} \right] ^2}
		+
		\frac{e^{- \left[\left( \frac{5}{4}+i\right) \frac{\pi}{t} \right] ^2}}{ \left[\left( \frac{5}{4}+i\right) \frac{\pi}{t} \right] ^2}\right] \nonumber\\
		&\le \left[ \left[\left( \frac{5}{4}+i\right) \frac{\pi}{t} \right] ^2- \left[\left( \frac{3}{4}+i\right) \frac{\pi}{t} \right] ^2\right]
		\frac{e^{- \left[\left( \frac{3}{4}+i\right) \frac{\pi}{t} \right] ^2}}{ \left[\left( \frac{3}{4}+i\right) \frac{\pi}{t} \right] ^2}.
		\end{align}
		Then, combining \eqref{mian2} and \eqref{mian3}, we have
		\begin{align}\label{mian4}
		B_i&\le \frac{\left[ \left[\left( \frac{5}{4}+i\right) \frac{\pi}{t} \right] ^2- \left[\left( \frac{3}{4}+i\right) \frac{\pi}{t} \right] ^2\right]}{\left[ \left[\left( \frac{3}{4}+i\right) \frac{\pi}{t} \right] ^2- \left[\left( \frac{1}{4}+i\right) \frac{\pi}{t} \right] ^2\right] } A_i
		=\left(1+\frac{1}{1+2i} \right) A_i
		\le 2 A_i.
		\end{align}
		By summing \eqref{mian4} with respect to $i$ from $0$ to $\infty$,  one can obtain the following estimate
		\begin{equation}
		\sum_{i=0}^\infty \left(A_i+B_i \right) \le 3 \sum_{i=0}^\infty A_i.
		\end{equation}
		Therefore, \eqref{mian1} is true and  \eqref{sl3} can become
		\begin{align}\label{sl4}
		K_1(t)
		& \ge \sum_{i=0}^{\infty}\int_{a_i}^{b_i} \frac{e^{-r^2} }{r} dr
		\ge \frac{1}{3}\int_{a_0}^{\infty} \frac{e^{-r^2} }{r} dr\nonumber\\
		&\ge  \frac{1}{3}\int_{\left( \frac{\pi}{4t}\right)^2 }^{1} \frac{e^{-r^2} }{r} dr
		\ge  \frac{e^{-1}}{3}\int_{\left( \frac{\pi}{4t}\right)^2 }^{1} \frac{1}{r} dr\nonumber\\
		&\ge  \frac{2e^{-1}}{3} \left(\log t+\log 4-\log \pi \right).
		\end{align}
		Finally, together with  (\ref{sl4}) and (\ref{sl2}), we obtain
		\begin{align}\label{sl5}
		\|\hat{u}\|_2^2
		\ge \frac{e^{-1}P^2}{6}  \left(\log t+\log 4-\log \pi \right) - \frac{C_1^2}{2}  \|u_1\|_{{1,1}} ^2 -\|u_0\|_2^2.
		\end{align}
		Thus,  there exists a positive real number $t_0$, one has the following blow-up property of the solution when $P \neq 0$
		\begin{align}\label{sl00}
		\|\hat{u}\|_2^2
		\ge \frac{P^2}{9 e^2}\log t,
		\end{align}
		for all $t \ge t_0$, where $t_0 > 0$ depends on  $\|u_1\|_{{1,1}} $ and $\|u_0\|_{2} $. The proof of the Lemma \ref{le2} is complete.
	\end{proof}

	\subsection{$L^2$-upper bound estimates of the solution}
	In this subsection, we derive upper bound estimates of $\|u(t,\cdot)\|_2$ as $t \to \infty$ by using Fourier splitting method.
	\begin{lemma}\label{le5}
		Let $n = 1$ and   $[u_0, u_1]\in L^2 (\mathbb{R}) \times  L^{1}(\mathbb{R}).$
		Then we have
		
		{\rm (1)} If $s\in (\frac{1}{2},1)$, then $$\|\hat{u}(t,\cdot)\|_2\le \left( \frac{4s}{2s-1}   \right) ^{\frac{1}{2}}\left(\|u_0\|_2+\|u_1\|_1 \right)  t^{1-\frac{1}{2s}},~~~~t\gg 1.$$
		
		{\rm (2)}  If $s=\frac{1}{2}$ and $u_1\in L^2 (\mathbb{R})$, then $$\|\hat{u}(t,\cdot)\|_2\le 2\left(\|u_0\|_2+\|u_1\|_2+\|u_1\|_1 \right) \sqrt{\log t},~~~~t\gg 1.$$
	\end{lemma}
	
	\begin{proof}
		First, applying  \eqref{ode1} and Cauchy's inequality, we have
		\begin{align}\label{ss1}
		\|\hat{u}\|_2 ^2
		&= \int_{\mathbb{R}^n} \left| \frac{\sin\left(t|\xi|^s \right) }{|\xi|^s}\hat{u}_1(\xi)+\cos\left(t|\xi|^s \right)\hat{u}_0(\xi)\right| ^2 d\xi\nonumber\\
		& \le 2\int_{\mathbb{R}^n}  \frac{\sin^2\left(t|\xi|^s \right) }{|\xi|^{2s}}\left|\hat{u}_1(\xi)\right| ^2 d\xi
		+2 \int_{\mathbb{R}^n}  \cos^2\left(t|\xi|^s \right) \left|\hat{u}_0(\xi)\right| ^2 d\xi\nonumber\\
		& \le 2\int_{\mathbb{R}^n}  \frac{\sin^2\left(t|\xi|^s \right) }{|\xi|^{2s}}\left|\hat{u}_1(\xi)\right| ^2 d\xi
		+2\|u_0\|_2^2 \nonumber\\
		&=: 2K_1(t) +2\|u_0\|_2^2.
		\end{align}
		In the sequel, we estimate the first term  $K_1(t)$ of the right-hand side of \eqref{ss1}
		\begin{align}\label{k1}
		K_1(t) = &\int_{|\xi|\le t^{-\frac{1}{s}}}\frac{\sin^2\left(t|\xi|^s \right) }{|\xi|^{2s}}\left|\hat{u}_1(\xi)\right| ^2 d\xi
		+\int_{ |\xi| \ge t^{-\frac{1}{s}}} \frac{\sin^2\left(t|\xi|^s \right) }{|\xi|^{2s}}\left|\hat{u}_1(\xi)\right| ^2 d\xi\nonumber\\
		\le& \; I^2 t^2 \|u_1\|_{1}^2 \int_{|\xi|\le t^{-\frac{1}{s}}} d\xi
		+ \|u_1\|_{1}^2 \int_{ |\xi| \ge t^{-\frac{1}{s}}} \frac{1 }{|\xi|^{2s}} d\xi  \nonumber\\
		\le&  \;\frac{4s}{2s-1} \|u_1\|_{1}^2 t^{2-\frac{1}{s}} ,
		\end{align}
		where $s \in (\frac{1}{2},1)$ and
		\begin{equation*}\label{I}
		I:=\sup\limits_{\theta\neq 0}\left|\frac{\sin \theta}{\theta} \right| \le 1.
		\end{equation*}
		Substituting the estimate \eqref{k1} into the right-hand side of \eqref{ss1}, we obtain in the case $n = 1$ and $s \in(\frac{1}{2},1)$
		\begin{align}\label{ss2}
		\|\hat{u}\|_2 ^2 \le \frac{4s}{2s-1} \left(\|u_0\|^2_2+\|u_1\|^2_1 \right) t^{2-\frac{1}{s}},~~~~t\gg 1.
		\end{align}

		Obviously, from \eqref{ss2} it can be seen that when $n=1$ and $s=\frac{1}{2}$, the polynomial blow-up upper bound for the $L^2$ norm of the solution can not be obtained,
		thus we try to find a more refined division of the frequency domain to obtain the upper bound estimate of $K_1(t)$, and we arrive at
		\begin{align}\label{su1}
		K_1(t) = &\int_{|\xi|\le t^{-2}}\frac{\sin^2\left(t|\xi|^{\frac{1}{2}} \right) }{|\xi|}\left|\hat{u}_1(\xi)\right| ^2 d\xi
		+\int_{ t^{-2} \le |\xi|\le (\log t)^{-1}} \frac{\sin^2\left(t|\xi|^{\frac{1}{2}} \right) }{|\xi|}\left|\hat{u}_1(\xi)\right| ^2 d\xi\nonumber\\
		&\qquad+\int_{|\xi|\ge (\log t)^{-1}} \frac{\sin^2\left(t|\xi|^{\frac{1}{2}} \right) }{|\xi|}\left|\hat{u}_1(\xi)\right| ^2 d\xi\nonumber\\
		\le& \; \|u_1\|_{1}^2 I^2 t^2 \int_{|\xi|\le t^{-2}}   d\xi
		+ \|u_1\|_{1}^2 \int_{ t^{-2} \le |\xi|\le (\log t)^{-1}} \frac{1 }{|\xi|} d\xi
		+\int_{|\xi|\ge (\log t)^{-1}} \frac{1}{|\xi|}\left|\hat{u}_1(\xi)\right| ^2 d\xi  \nonumber\\
		\le&  \;2 \|u_1\|_{1}^2
		+ 2\|u_1\|_{1}^2 \left( 2\log t-\log (\log t) \right)
		+\|u_1\|_{2}^2 \log t.
		\end{align}
		Finally, substituting \eqref{su1} into \eqref{ss1}, we obtain in the case $n = 1$ and $s =\frac{1}{2}$		
		\begin{align}
		\|\hat{u}\|_2 ^2 \le  4\left(\|u_0\|^2_2+\|u_1\|^2_2+\|u_1\|^2_1 \right) \log t.
		\end{align}
		The proof of the Lemma \ref{le5} is complete.
	\end{proof}

	Finally, the proofs of Theorems \ref{th1} and \ref{th2} are direct consequences of Lemmas \ref{le1}, \ref{le2},  \ref{le5}, and the Plancherel Theorem.

\end{document}